\documentclass[11pt]{article}
\usepackage [dvips]{graphics}
\usepackage[centertags]{amsmath}
\usepackage{amsfonts}
\usepackage{amssymb}
\usepackage{amsthm}
\usepackage{newlfont}
\usepackage{stmaryrd}
\usepackage{color,epsfig}
\usepackage{amsfonts,graphicx,psfrag}
\usepackage{graphicx}
\usepackage{float}

\usepackage{txfonts}
\usepackage{mathrsfs}

\usepackage{amssymb, amsmath,amsmath,latexsym,amssymb,amsfonts,amsbsy, amsthm,mathtools,graphicx,CJKutf8,CJKnumb,CJKulem,color}
\usepackage{graphicx,epsfig}
\usepackage{graphicx}
\usepackage{amssymb}
\usepackage{graphicx}
\usepackage{graphicx,epsfig}
\usepackage{amsmath}
\usepackage{mathrsfs}
\usepackage{amssymb}

\usepackage{amssymb,mathrsfs,amsfonts,amsmath,amsbsy,tikz}\usepackage{fontenc}\usepackage{textcomp}

\numberwithin{equation}{section}

\allowdisplaybreaks

\let\z=\zeta

\let\D=\Delta

\let\ep=\epsilon

\def\cA{{\cal A}}

\def\cX{{\cal X}}

\def\cZ{{\cal Z}}

\def\no{\noindent}

\newcommand{\beq}{\begin{equation}}
\newcommand{\eeq}{\end{equation}}
\newcommand{\ben}{\begin{eqnarray}}
\newcommand{\een}{\end{eqnarray}}
\newcommand{\beno}{\begin{eqnarray*}}
\newcommand{\eeno}{\end{eqnarray*}}

\newtheorem{theorem}{Theorem}[section]
\newtheorem{definition}[theorem]{Definition}
\newtheorem{lemma}[theorem]{Lemma}
\newtheorem{proposition}[theorem]{Proposition}

\newtheorem{remark}[theorem]{Remark}

\theoremstyle{plain}
\newtheorem{thm}{Theorem}[section]
\newtheorem{cor}[thm]{Corollary}
\newtheorem{lem}[thm]{Lemma}
\newtheorem{prop}[thm]{Proposition}

\newtheorem{rem}[thm]{Remark}
\newtheorem{defi}[thm]{Definition}

\def\sqr#1#2{{\vcenter{\vbox{\hrule height.#2pt
              \hbox{\vrule width.#2pt height#1pt \kern#1pt \vrule
width.#2pt}
              \hrule height.#2pt}}}}
\def\be{\begin{equation}}
\def\ee{\end{equation}}

\def\ga{{\gamma}}

\def\ep{{\epsilon}}

\def\Sp{{\mathrm {Sp}}}

\def\ga{{\gamma}}

\def\D{\Delta}

\def\cA{{\cal A}}

\def\cX{{\cal X}}

\def\cZ{{\cal Z}}

\def\no{\noindent}

\def\ss{\smallskip}

\def\bs{\bigskip}

\def\dim{\hbox{\rm dim$\,$}}

\def\({\Big (}
\def\){\Big )}
\def\[{\Big[}
\def\]{\Big]}

\def\be{\begin{equation}}
\def\bel{\begin{equation}\label}
\def\ee{\end{equation}}
\def\bea{\begin{eqnarray}}
\def\eea{\end{eqnarray}}
\def\bt{\begin{theorem}}
\def\et{\end{theorem}}
\def\bc{\begin{corollary}}
\def\ec{\end{corollary}}
\def\bl{\begin{lemma}}
\def\el{\end{lemma}}
\def\bp{\begin{proposition}}
\def\ep{\end{proposition}}
\def\br{\begin{remark}}
\def\er{\end{remark}}
\def\ba{\begin{array}}
\def\ea{\end{array}}
\def\bd{\begin{definition}}
\def\ed{\end{definition}}

\makeatletter
   
   \@addtoreset{equation}{section}
\makeatother

\begin{document}

\title{\bf On the $j$-th  Eigenvalue of Sturm-Liouville  Problem and the Maslov Index}
\author
{Xijun Hu$^1$\thanks{ E-mail: xjhu@sdu.edu.cn }  \quad  Lei Liu$^1$\thanks{ E-mail: 201511242@mail.sdu.edu.cn} \quad Li Wu$^1$\thanks{
 E-mail: 201790000005@sdu.edu.cn} \quad Hao Zhu$^2$\thanks{
 E-mail:  haozhu@nankai.edu.cn }    \\ \\
$^1$ Department of Mathematics, Shandong University\\
Jinan, Shandong 250100, P. R. China\\
$^2$ Chern Institute of Mathematics, Nankai University\\[2pt]
					  Tianjin, 300071, P. R. China
}

\date{}
\maketitle
\begin{abstract}
In the previous papers \cite{HLWZ,KWZ}, the jump phenomena of the $j$-th eigenvalue were  completely characterized for Sturm-Liouville problems.
In this paper,
we show that the  jump number of these eigenvalue branches is exactly the Maslov index for the path of corresponding boundary conditions.
 Then we  determine the sharp range of the $j$-th eigenvalue on  each layer of the space of boundary conditions. Finally, we prove that the graph of monodromy matrix  tends to the Dirichlet boundary condition  as the spectral parameter goes to $-\infty$.
\end{abstract}

\bs

\no{\bf AMS Subject Classification:}  34B24, 34L15, 53D12, 34B09.

\bs

\no{\bf Key Words}.  Sturm-Liouville problem,  eigenvalue, boundary condition,  Maslov index,  Lagrangian subspace.
\section{Introduction}

 In this paper, we consider   the  Sturm-Liouville  problem
 \begin{equation}\label{eq:S-L equation}
 -\frac{d}{dt}\left(P(t)\frac{d}{dt}{x}(t)+Q(t)x(t)\right)+Q(t)^{T}\frac{d}{dt}{x}(t)+R(t)x(t)=\lambda D(t) x(t),
  \end{equation}
 where $P, Q\in H^1([0,T],L(n))$, $R, D\in C([0,T],L(n))$,  $P(t), D(t)$ are positive  definite, and $P(t), R(t), D(t)$ are symmetric for all $t\in [0, T]$. Here $L(n)$ is the set of $n\times n$ real-valued matrices and $Q(t)^T$ is the transpose of $Q(t)$.
 We describe a self-adjoint boundary condition of \eqref{eq:S-L equation} by a Lagrangian subspace of $\mathbf{C}^{2n}\oplus\mathbf{C}^{2n}$ as follows.
Consider $(\mathbf{C}^{2n},\omega_n)$  as a complex symplectic vector space with the symplectic form $\omega_n(x,y)=\langle J_nx,y\rangle$ for any
 $x,y\in \mathbf{C}^{2n},$
where   $\langle\cdot,\cdot\rangle$ is the standard Hermitian inner product in $\mathbf{C}^{2n}$ and $J_n=\begin{bmatrix} 0&-I_{n}\\I_{n}&0 \end{bmatrix}$.
   Denote $$(\mathcal{V}, \Omega)=(\mathbf{C}^{2n}\oplus\mathbf{C}^{2n},-\omega_n\oplus\omega_n),$$
 which is a $4n$-dimensional symplectic space.
  A  subspace $\Lambda\subset\mathcal{V}$ is called  Lagrangian if $\Omega\mid _{\Lambda}=0 $ and $\dim_{\mathbf{C}}\Lambda=\frac{1}{2}\dim_{\mathbf{C}}\mathcal{V}=2n$.   Denote
  the set of Lagrangian subspaces by $\text{Lag}(\mathcal{V}, \Omega)$. Then   $\text{Lag}(\mathcal{V}, \Omega)$ is a compact metric space \cite{BoZh14}. Let $\dot{x}=\frac{d}{dt}x$,
  $y(t)=P(t)\dot{x}(t)+Q(t)x(t),$ and $z(t)=(y(t)^T,x(t)^T)^T$. Then any self-adjoint  boundary condition can be written as
 \begin{equation}\label{eq:original general boundary condition}
\begin{bmatrix}
		z(0)\\z(T)
		\end{bmatrix}\in\Lambda_0,
 \end{equation}
 where $\Lambda_0\in \text{Lag}(\mathcal{V}, \Omega)$.
In particular, the Neumann and Dirichlet boundary conditions are given by
$$\Lambda_N=\left\{\begin{bmatrix}
		z(0)\\z(T)
		\end{bmatrix}\in\mathbf{C}^{4n}: y(0)=y(T)=0\right\}, \quad  \Lambda_D=\left\{\begin{bmatrix}
		z(0)\\z(T)
		\end{bmatrix}\in\mathbf{C}^{4n}: x(0)=x(T)=0 \right\},$$
respectively. Moreover, $\text{Lag}(\mathcal{V}, \Omega)$ is exactly the space of self-adjoint boundary conditions.

The formal differential operator corresponding to (\ref{eq:S-L equation}) is $$\mathcal{A}:=-\frac{d}{dt}\left(P\frac{d}{dt}+Q\right)+Q^{T}\frac{d}{dt}+R.$$
Define an operator $\mathcal{A}_{\Lambda_0}x:=\mathcal{A}x$ on $L^2([0,T], \mathbf{C}^{n} )$ with the domain
$$E_{\Lambda_0}(0,T):=\left\{x\in H^{2}([0,T], \mathbf{C}^{n}): \begin{bmatrix}
		z(0)\\z(T)
		\end{bmatrix}\in\Lambda_0  \right\}.$$
Then $\mathcal{A}_{\Lambda_0}$ is a self-adjoint operator.
Note that  \eqref{eq:S-L equation}--\eqref{eq:original general boundary condition} are equivalent to
 $$
 \mathcal{A}_{\Lambda_0}x=\lambda D x ,    $$
  which can be written as  $$  D^{-\frac{1}{2}}\mathcal{A}_{\Lambda_0} D^{-\frac{1}{2}}y=\lambda y ,    $$
  where $y=D^{\frac{1}{2}}x$. Therefore,  without loss of generality, we always assume that $D=I_n$.

  The spectrum of $
 \mathcal{A}_{\Lambda_0}$ is bounded from below and consists of discrete eigenvalues, which are listed as follows:
  $$\lambda_1(\Lambda_0)\leq \cdots\leq \lambda_j(\Lambda_0)\leq\cdots$$
   counting multiplicities, with   $\lambda_j(\Lambda_0)\to \infty$ as $j\to\infty$. Thus the $j$-th  eigenvalue $\lambda_j$ can be regarded as a function $\lambda_j:\text{Lag}(\mathcal{V}, \Omega)\to\mathbf{R}$ in the sequel.  $\lambda_j$ is  not always continuously dependent on $\Lambda_0$. Recently, Kong, Wu and Zettl completely characterized the discontinuity of $\lambda_j$ for $1$-dimensional case in \cite{KWZ}, while  we  characterized it for $n$-dimensional case with $n\geq2$ in \cite{HLWZ}. In fact, discontinuity  may occur only at such boundary condition $\Lambda_0$ that $$\lim_{s\to 0^\pm}\dim(\Lambda_s\cap\Lambda_D)\neq\dim(\Lambda_0\cap\Lambda_D),$$
where $\Lambda_s$, $s\in[-\epsilon, +\epsilon]$, is  a continuous path in $\text{Lag}(\mathcal{V},\Omega)$.
Near such boundary condition $\Lambda_0$,  the $j$-th eigenvalue always jumps in certain directions.
     In this paper, as a continuous work of  \cite{HLWZ,KWZ}, we use the Maslov index to count the jump number of $\lambda_j$.
  For readers' convenience, we give a brief introduction of the Maslov index in Section 2.

 To describe the discontinuity in the framework of Lagrangian subspaces,  we consider a continuous path $\Lambda_s\in \text{Lag}(\mathcal{V},\Omega)$, $s\in[-\epsilon, +\epsilon]$, with the isolated singularity at $s=0$.  More precisely,
  \bea \dim \Lambda_s\cap \Lambda_D=c_-\ \textrm{for}\ s\in[-\epsilon, 0), \  \ \dim \Lambda_0\cap \Lambda_D=c_0, \   \ \dim \Lambda_s\cap \Lambda_D=c_+ \ \textrm{for}\ s\in(0,+\epsilon] , \label{1.4} \eea  and
  \bea \mu(\Lambda_D, \Lambda_s, s\in[-\epsilon,0])=-k_-, \quad   \mu(\Lambda_D, \Lambda_s, s\in[0, +\epsilon])=k_+,   \label{1.5}  \eea
   where  $\mu(\cdot,\cdot,\cdot)$ is the Maslov index, see Definition \ref{defi2.1}.

 For convenience, we  set $\lambda_j=-\infty$ for  $j\leq 0$. Our first main result is stated as follows.

  \begin{thm}\label{thm: Maslov index continu}
 Let $\Lambda_s\in  \text{Lag}(\mathcal{V}, \Omega)$, $s\in[-\epsilon,\epsilon]$, be a continuous path, and satisfy \eqref{1.4}--\eqref{1.5}.   Then $\lambda_j(\Lambda_\cdot)$ is continuous on $s\in[-\epsilon,0)\cup(0, +\epsilon]$ and
  \begin{equation} \label{jump12}
 \lim_{s\to 0^-} \lambda_j(\Lambda_s)=\lambda_{j-k_-}(\Lambda_0), \quad \lim_{s\to 0^+} \lambda_j(\Lambda_s)=\lambda_{j-k_+}(\Lambda_0).
 \end{equation}
\end{thm}

 By Corollary \ref{cor2.5}, $k_\pm$ are non-negative.  Our new contribution is that  the jump number $n^+-n_0^+$ in Theorem 7.1 of \cite{HLWZ} is exactly the Maslov index $k_\pm$ in Theorem \ref{thm: Maslov index continu}.
 The  idea of the proof of Theorem \ref{thm: Maslov index continu}  is to express the  $j$-th  eigenvalue by the index form. We  refer the readers to  \cite{HWY} for the introduction of index form. Then we study the monotone property of index form and give proper estimates for the eigenvalues.

The  range of the $j$-th eigenvalue $\lambda_j$ on the whole space of  boundary conditions   was given in Theorem 4.1 of \cite{KWZ} for $1$-dimensional Sturm-Liouville problems. To the best of our knowledge, there are no results for  high dimensional case. Define  the  $r$-th layer on $\text{Lag}(\mathcal{V},\Omega)$ to be
\begin{equation*}
\Sigma_{r}=\{\Lambda\in \text{Lag}(\mathcal{V},\Omega):\dim \left(\Lambda\cap\Lambda_D\right)=r\},
\end{equation*}
 where  $0\leq r\leq 2n$.
Our second result is  to determine the sharp range of  $\lambda_j$ on each layer of  $\text{Lag}(\mathcal{V},\Omega)$ for $n$-dimensional Sturm-Liouville problems.

\begin{thm}\label{thm: eigen range} Fix any $j\geq1$ and $0\leq r\leq 2n$.
Let
$$\lambda_{j-(2n-r)-b_1}(\Lambda_D)=\cdots=\lambda_{j-(2n-r)}(\Lambda_D)=\cdots=\lambda_{j-(2n-r)+b_2}(\Lambda_D)$$
with multiplicity to be $b_1+b_2+1$, and
$$\lambda_{j-c_1}(\Lambda_D)=\cdots=\lambda_{j}(\Lambda_D)=\cdots=\lambda_{j+c_2}(\Lambda_D)$$
with multiplicity to be $c_1+c_2+1$, where $b_i, c_i\geq 0$, $i=1,2$.
If  $j\leq 2n-r$, then we have two cases.

Case 1:  $r\leq c_2$.
 $$\lambda_j(\Sigma_r)=(-\infty,\lambda_j(\Lambda_D)).$$

Case 2: $r> c_2$.
$$\lambda_j(\Sigma_r)=(-\infty,\lambda_j(\Lambda_D)].$$

If  $j> 2n-r$, then we have four cases.

Case 1:  $r\leq \min\{b_1,c_2\}.$
 $$\lambda_j(\Sigma_r)=(\lambda_{j-(2n-r)}(\Lambda_D),\lambda_j(\Lambda_D)).$$

 Case 2:  $c_2<r\leq b_1.$
 $$\lambda_j(\Sigma_r)=(\lambda_{j-(2n-r)}(\Lambda_D),\lambda_j(\Lambda_D)].$$

Case 3:  $b_1<r\leq c_2$.
 $$\lambda_j(\Sigma_r)=[\lambda_{j-(2n-r)}(\Lambda_D),\lambda_j(\Lambda_D)).$$

Case 4:  $r>\max\{b_1,c_2\}$.
 $$\lambda_j(\Sigma_r)=[\lambda_{j-(2n-r)}(\Lambda_D),\lambda_j(\Lambda_D)].$$

\end{thm}

Theorem \ref{thm: eigen range} indicates that the ``left-multiplicity" $b_1$ of $\lambda_{j-(2n-r)}(\Lambda_D)$, the ``right-multiplicity" $c_2$ of $\lambda_j(\Lambda_D)$ and the layer's number $r$ determine whether the endpoints $\lambda_{j-(2n-r)}(\Lambda_D), \lambda_j(\Lambda_D)\in \lambda_j(\Sigma_r)$ or not.
Then the range of $\lambda_j$ on the whole space $\text{Lag}(\mathcal{V},\Omega)$ and on the $0$-th layer $\Sigma_0$  is a direct consequence.

\begin{cor} \label{whole space} For any $j\geq1$, we have

(1) $\lambda_j(\text{Lag}(\mathcal{V},\Omega))=(\lambda_{j-2n}(\Lambda_D), \lambda_j(\Lambda_D)]$;

(2)  $\lambda_j(\Sigma_{0})=(\lambda_{j-2n}(\Lambda_D), \lambda_j(\Lambda_D))$.
\end{cor}

 Corollary \ref{whole space} (1)  generalizes Theorem 4.1 in \cite{KWZ} for $1$-dimensional result to any dimension.  As an example of  Theorem \ref{thm: eigen range},  we provide the sharp range of $\lambda_j$ on each layer for $1$-dimensional case, which is more  accurate  than the conclusions  in \cite{KWZ}:

\begin{cor}
\label{rem-1-dim-case}
For any given
$1$-dimensional Sturm-Liouville equation, we have for any $j\geq1$,

(1) $\lambda_j(\Sigma_{0})=(\lambda_{j-2}(\Lambda_D), $ $ \lambda_j(\Lambda_D))$;

(2) $\lambda_1(\Sigma_{1})=(-\infty, \lambda_1(\Lambda_D)]$, and  $\lambda_j(\Sigma_{1})=[\lambda_{j-1}(\Lambda_D), \lambda_j(\Lambda_D)]$ for $j\geq2$;

(3) $\lambda_j(\Sigma_{2})=\{\lambda_{j}(\Lambda_D)\}$.
\end{cor}

Since the multiplicity of an eigenvalue of $\mathcal{A}_{\Lambda_D}$ is at most $n$, we get the following result.
\begin{cor}
\label{independent of multiplicity}
 Let $n\leq r\leq 2n$. Then $\lambda_j(\Sigma_{r})=(-\infty, \lambda_j(\Lambda_D)]$ for any $1\leq j\leq 2n-r$, and
$\lambda_j(\Sigma_{r})=[\lambda_{j-(2n-r)}(\Lambda_D),$ $ \lambda_j(\Lambda_D)]$ for any $j> 2n-r$.
\end{cor}
By Theorem  \ref{thm: eigen range}, we also get the following interesting fact.
\begin{cor}\label{rem-multiple-eigenvalue}
 Let   $\lambda_{{j_0}-(r_0-1)} ({\Lambda_D})=\cdots=\lambda_{j_0} ({\Lambda_D})$ with multiplicity to be $r_0$, where $1\leq r_0\leq n$ and $j_0\geq r_0$. Then  $\lambda_{j_0}(\Lambda)\equiv\lambda_{j_0}(\Lambda_D)$ for all $\Lambda\in\Sigma_{r}$, where $2n-r_0+1\leq r\leq 2n$.
\end{cor}

By the standard Legendre transformation, equation (\ref{eq:S-L equation}) with $D=I_n$ becomes
\begin{eqnarray}
\dot{z}(t)=J_n\mathcal{B}_\lambda(t)z(t), \label{eq:Hamilton system}
\end{eqnarray}
where
\begin{equation*}
\mathcal{B}_\lambda(t)=\begin{bmatrix} P^{-1}(t) &- P^{-1}(t)Q(t)\\-Q^T(t)P^{-1}(t)&Q^{T}(t)P^{-1}(t)Q(t)-R(t)+\lambda I_n \end{bmatrix}.
\end{equation*}
Let $\gamma_\lambda(t)$,  $t\in[0,T]$, be the fundamental solution of  \eqref{eq:Hamilton system}, that is, for   $\lambda\in\mathbf{C}$,
\begin{equation}\label{fundamental solution} \dot{\gamma}_\lambda(t)=J_n\mathcal{B}_\lambda(t)\gamma_\lambda(t), \quad \gamma_\lambda(0)=I_{2n}.\end{equation}
 It is well-known that for any $t\in[0,T]$,
$$ \gamma_\lambda(t)\in \Sp(2n):=\{M\in GL(\mathbf{R}^{2n}):M^TJ_nM=J_n  \}. $$

Since $\gamma_\lambda(T)\in\Sp(2n)$, it is obvious that $Gr(\gamma_\lambda(T))\in  \text{Lag}(\mathcal{V}, \Omega)$, where
   $$Gr(\gamma_\lambda(T)):=\{(x,\gamma_\lambda(T)x), x\in\mathbf{C}^{2n} \}$$
   is  the graph of $\gamma_\lambda(T)$.

\begin{thm}\label{thm:lim to Lag} Under the above notation, we have
 \begin{equation}\label{thm:lim to Lag1}
 \lim_{\lambda\to-\infty}Gr(\ga_\lambda(T))=\Lambda_D.
 \end{equation}
Furthermore, let $\lambda(s)=\tan (s)$ for $s\in [-\pi/2, +\pi/2]$, then
\begin{equation}\label{thm:lim to Lag2}
\mu(\Lambda_D,Gr(\gamma_{\lambda(s)}(T)), s\in[-\pi/2,-\pi/2+\epsilon])= 2n
\end{equation}
 for $\epsilon>0$ small enough.
\end{thm}
From \eqref{thm:lim to Lag1}, we define  $Gr(\ga_{-\infty}(T))=Gr(\ga_{\lambda(-\pi/2)}(T))=\Lambda_D$ and thus  $Gr(\ga_{\lambda(s)}(T)), s\in[-\pi/2, -\pi/2+\epsilon]$ is a continuous path in $\text{Lag}(\mathcal{V}, \Omega)$.
  \eqref{thm:lim to Lag2} implies that    $Gr(\gamma_{\lambda(s)}(T)), s\in[-\pi/2, -\pi/2+\epsilon]$ can be considered as a positive path in the sense that all the  eigenvalues of the corresponding path of  unitary matrices are rotated  counterclockwise. Here we use the definition of Maslov index in
Remark \ref{rm:def_mas_2}.


The rest of this  paper is organized as follows. In Section $2$, we briefly review the Maslov index theory for Lagrangian subspaces. In Section 3,
we give the proof of Theorem \ref{thm: Maslov index continu}. Theorem \ref{thm: eigen range} is shown in Section 4. In Section 5, we provide the proof of Theorem \ref{thm:lim to Lag}.

\section{Maslov  Index  for the Lagrangian subspaces}

  In this section,  we  briefly introduce the general Maslov index theory for the Lagrangian subspaces. Then we apply it to our framework for Sturm-Liouville problems. For this theory, we refer the readers to \cite{Ar,BoZh14,CLM,Lon4} and references therein.

  The Lagrangian frame of a given $\Lambda\in \text{Lag}( \mathbf{C}^{2m},\omega_m )$   is defined by an injective linear map $\mathcal{Z}:\mathbf{C}^{m}\rightarrow \Lambda$ with the form $\mathcal{Z}=\begin{bmatrix}X\\Y \end{bmatrix}$, where $X$ and $Y$ are $m\times m$ complex matrices such that $X^*Y=Y^*X$ and $rank(\mathcal{Z})=m$. Here $X^*$ is  the conjugate transpose of  $X$. 
  The Lagrangian subspace is represented by a corresponding Lagrangian frame
   in the sequel.
		Clearly, $P_+=\begin{bmatrix}
		I_{m}\\iI_{m}
		\end{bmatrix}$ is a bijection from $\mathbf{C}^{m}$ to $\Lambda^+=\ker(iJ_{m}-I_{2m})$ and
		$P_-=\begin{bmatrix}
		I_{m}\\-iI_{m}
		\end{bmatrix}$ is a bijection from $\mathbf{C}^{m}$ to $\Lambda^-=\ker(iJ_{m}+I_{2m})$.
	For any $v\in \mathbf{C}^{m}$,	we  decompose $\cZ v$ to be
		$$
		\cZ v=\begin{bmatrix}
			(X-iY)v/2\\ (Y+iX)v/2
		\end{bmatrix}
		+
		\begin{bmatrix}
			(X+iY)v/2\\ (Y-iX)v/2
		\end{bmatrix}=P_+((X-iY)v/2)+P_-((X+iY)v/2).
		$$
Then we get a unitary operator
		 \begin{equation*}
		W =P_-(X+iY)(X-iY)^{-1}P_+^{-1}
		 \end{equation*}
from $\Lambda^+$ to $\Lambda^-$.
	Correspondingly, $P_-^{-1}W P_+=(X+iY)(X-iY)^{-1}$ is an $m\times m$ unitary  matrix.
  So we  define a  map $\mathcal{U}: \text{Lag}( \mathbf{C}^{2m},\omega_m )\rightarrow \mathbf{U}(m)$ as follows:
  \begin{equation} \label{def-f}\mathcal{U}(\Lambda)=(X+iY)(X-iY)^{-1}\in\mathbf{U}(m).  \end{equation}
 Note that  $\mathcal{U}$ is  a homeomorphic (isomorphic) map \cite{HWY} and  $\mathcal{U}(\Lambda)$ is independent of the choice of frame.
   Let $\Lambda_k\in \text{Lag}(\mathbf{C}^{2m},\omega_m)$ with   Lagrangian frame to be $\begin{bmatrix}X_k\\Y_k \end{bmatrix}$, and $W_k =P_-(X_k+iY_k)(X_k-iY_k)^{-1}P_+^{-1}$, $k=1,2$. Then	it follows that
		$\mathcal{U}(\Lambda_2)^{-1}\mathcal{U}(\Lambda_1)=P_+^{-1}W_{2}^{-1}W_1 P_+$ and the spectrum of $\mathcal{U}(\Lambda_2)^{-1}\mathcal{U}(\Lambda_1)$ is the same with that of $W_2^{-1}W_1$.
  The metric of  Lagrangian subspaces is defined as the metric of corresponding unitary matrices on $\mathbf{U}(m)$, that is,
  \begin{equation*}  \text{dist} (\Lambda_1, \Lambda_2):=\|\mathcal{U}(\Lambda_1)-\mathcal{U}(\Lambda_2)\|,   \end{equation*}
   where $\|\cdot\|$ is the operator norm. Moreover,
$
\dim \left(\Lambda_{1}\cap \Lambda_{2}\right)=\dim \left(\ker(\mathcal{U}(\Lambda_{2})^{-1}\mathcal{U}(\Lambda_{1})-I_{m})\right).
$
For any fixed $U_{0}\in \mathbf{U}(m)$, the singular cycle $\Sigma_{U_{0}}$ of $U_{0}$ is defined  by
$$
 \Sigma_{U_{0}}=\{U\in \mathbf{U}(m)\mid \det(U_{0}^{-1}U-I_{m})=0\}.
$$
 Now we introduce the  Maslov index. Consider  a continuous path $U_{t}, t\in [a,b]$, in $\mathbf{U}(m)$ and the  small perturbation
 $e^{is}U_{t}$, $s\in[-\varepsilon,\varepsilon]$. For any fixed $t_{0}\in [a,b]$, the path $e^{is}U_{t_{0}}$, $s\in[-\varepsilon,\varepsilon]$, is transversal to $\Sigma_{U_{0}}$. Furthermore, $e^{-is_{0}}U_{a}, e^{-is_{0}}U_{b}\notin \Sigma_{U_{0}}$ for
 any $s_{0}>0$ sufficiently small.
Thus the intersection number
  $[e^{-is_{0}}U_{t}:\Sigma_{U_{0}}]$ can be  well-defined. Then we give the concept of Maslov index:
\begin{defi}\label{defi2.1}
Let $\Lambda(t), t\in[a,b]$, be a continuous path in $\text{Lag}(\mathbf{C}^{2m},\omega_m)$ and $\Lambda_{0}\in \text{Lag}(\mathbf{C}^{2m},\omega_m)$. Then the Maslov index is defined by
\begin{equation*}
\mu(\Lambda_{0},\Lambda(t),t\in[a,b]):=[e^{-is_{0}}\mathcal{U}(\Lambda(t)): \Sigma_{\mathcal{U}(\Lambda_{0})}],
\end{equation*}
where $s_{0}>0$ is sufficiently small.
\end{defi}

\begin{rem}\label{rm:def_mas_2} Definition \ref{defi2.1} is equivalent to the definition of Maslov index in  Section 2.2 of \cite{BoZh14}, which is  defined as follows.
		There are $m$ continuous functions $\theta_j\in C([a,b],\mathbf{R})$ such that $e^{i\theta_j(t)},1\leq j\leq m,$ are all the eigenvalues of $\mathcal{U}(\Lambda_0)^{-1}\mathcal{U}(\Lambda(t))$ (counting algebraic multiplicities).
		Denote by $[a]$ the integer part of $a\in \mathbf{R}$.
		Define $E(a)=-[-a]$.
		Then the Maslov index can be defined as
		 \begin{equation*}
		\mu(\Lambda_0,\Lambda(t),t\in[a,b])=\sum_{j=1}^{m}\left(E\left(\frac{\theta_j(b)}{2\pi}\right)-E\left(\frac{\theta_j(a)}{2\pi}\right)\right).
		 \end{equation*}
\end{rem}
Then we provide some properties of Maslov index and we refer to \cite{BoZh14} for the details.

{Property I (Reparametrization invariance)  Let
	$\phi:[c,d]\rightarrow [a,b]$ be a continuous and piecewise smooth
	function with $\phi(c)=a$, $\phi(d)=b$. Then
 \begin{equation*} \mu(\Lambda_1(t),
	\Lambda_2(t),t\in[a,b])=\mu(\Lambda_1(\phi(\tau)), \Lambda_2(\phi(\tau)),\tau\in[c,d]).  \end{equation*}
	
	Property II (Homotopy invariant with endpoints)  For two continuous
	families of Lagrangian paths $\Lambda_1(s,t)$, $\Lambda_2(s,t)$, $0\leq s\leq 1$, $a\leq
	t\leq b$, such that both $\dim \left(\Lambda_1(s,a)\cap \Lambda_2(s,a)\right)$ and $\dim
	\left(\Lambda_1(s,b)\cap \Lambda_2(s,b)\right)$ are constants, we have
 \begin{equation*} \mu(\Lambda_1(0,t),
	\Lambda_2(0,t),t\in[a,b])=\mu(\Lambda_1(1,t),\Lambda_2(1,t),t\in[a,b]).  \end{equation*}

	Property III (Path additivity)  If $a<c<b$, then
	 \begin{equation*} \mu(\Lambda_1(t),\Lambda_2(t),t\in[a,b])=\mu(\Lambda_1(t),\Lambda_2(t),t\in {[a,c]})+\mu(\Lambda_1(t),\Lambda_2(t),t\in{[c,b]}).
	 \end{equation*}

	Property IV (Symplectic invariance)  Let $\ga(t)$, $t\in[a,b],$ be a
	continuous path in $\Sp(2m)$. Then
 \begin{equation*} \mu(\Lambda_1(t),\Lambda_2(t),t\in[a,b])=\mu(\ga(t)\Lambda_1(t), \ga(t)\Lambda_2(t),t\in[a,b]). \end{equation*}

	Property V (Symplectic additivity)  Let $W_i$, $i=1,2$, be symplectic spaces,
	$\Lambda_1(\cdot),\Lambda_2(\cdot)\in C([a,b],$ $ \text{Lag}(W_1))$ and ${\Lambda}_3(\cdot),{\Lambda}_4(\cdot)\in C([a,b], \text{Lag}(W_2))$.
	Then we have  \begin{equation*}  \mu(\Lambda_1(t)\oplus{\Lambda}_3(t),\Lambda_2(t)\oplus {\Lambda}_4(t),t\in[a,b])= \mu(\Lambda_1(t),\Lambda_2(t),t\in[a,b])+
	\mu({\Lambda}_3(t),{\Lambda}_4(t),t\in[a,b]).
	 \end{equation*}}

Now we turn back to the framework for boundary conditions of Sturm-Liouville problems. Firstly, we  change the basis of    $(\mathcal{V}, \Omega)$ such that the symplectic structure becomes the standard form $\omega_{2n}$.
Recall that  $\Omega=-\omega_n\oplus\omega_n$ corresponds to the matrix
$$\mathcal{J}=\begin{bmatrix} -J_{n}&0\\0&J_{n} \end{bmatrix}.$$
Direct computation implies
 $$S \mathcal{J} S=J_{2n},$$
 where
 $$
S=\begin{bmatrix} -I_{n}&0&0&0\\0&0& I_{n}&0\\0&I_{n}&0&0\\0&0&0&I_{n} \end{bmatrix}.
$$
Under the new basis $\omega_{2n}$, the boundary condition \eqref{eq:original general boundary condition} becomes
 \begin{equation*}S\begin{bmatrix}
		z(0)\\z(T)
		\end{bmatrix}=\begin{bmatrix}
		-y(0)\\y(T)\\x(0)\\x(T)
		\end{bmatrix}\in \Lambda_0.
 \end{equation*}

 Next, we provide a Lagrangian frame of any given $\Lambda_0 \in \text{Lag}(\mathcal{V},\omega_{2n})$.  Let $V(\Lambda_0)$ be the subspace of $\Lambda_N$ defined by $$V(\Lambda_0)=(\Lambda_0+\Lambda_D)\cap\Lambda_N.$$
 Then $(x(0)^T,x(T)^T)^T\in V(\Lambda_0)$.
  Thanks to the splitting $\mathbf{C}^{2n}\cong \Lambda_N=V(\Lambda_0)\oplus V(\Lambda_0)^\perp$, we get $(-y(0)^T, y(T)^T)^T$ $=(-y_1(0)^T, y_1(T)^T)^T+(-y_2(0)^T, y_2(T)^T)^T$, where $(-y_1(0)^T, y_1(T)^T)^T\in J_{2n} V(\Lambda_0)^\perp$ and $(-y_2(0)^T, y_2(T)^T)^T\in J_{2n} V(\Lambda_0)$. Then we have a linear map $A$ from $V(\Lambda_0)$ to $J_{2n}V(\Lambda_0)$ such that $(-y_2(0)^T, y_2(T)^T)^T=A(x(0)^T,$ $x(T)^T)^T$.  $A$ is Hermitian since $\Lambda_0$ is Lagrangian.  By assuming $\dim V(\Lambda_0)=k_0  $,   we can choose a suitable basis of $\mathcal{V}$ such that
  \begin{equation}\label{eq:precise express of boundary condition}
  \begin{bmatrix}I_{2n-k_0}&0\\0&A\\0&0\\0&I_{k_0}\end{bmatrix}
  \end{equation}
 is a Lagrangian frame of $\Lambda_0 $  under the symplectic form $J_{2n}$. The left column corresponds to $\Lambda_0\cap\Lambda_D$ and the right column corresponds to $\left\{
   \begin{bmatrix}Au \\u \end{bmatrix} \ : \ u\in V(\Lambda_0)\right\}$.
In addition, $\begin{bmatrix} -I_n&0\\D_1&D_2\\0&I_n\\D_3&D_4 \end{bmatrix}$ is a frame of $Gr(\mathcal{D})$ for $\mathcal{D}=\begin{bmatrix}D_1 &D_2\\D_3 & D_4 \end{bmatrix}\in\Sp(2n)$.
By \eqref{def-f} we have
 $$\mathcal{U}(\Lambda_0)=I_{2n-k_0}\oplus U_A,  $$  where $U_A=(A+iI_{k_0})(A-iI_{k_0})^{-1}$.  Especially,
$$ \mathcal{U}(\Lambda_D)=I_{2n},\quad   \mathcal{U}(\Lambda_N)=-I_{2n}.   $$

Let  $\Lambda_s$, $s\in[\tau_0, \tau_1],$ be a continuous path with $\tau_1-\tau_0$ small enough. From the definition of Maslov index, we will show that  $\mu(\Lambda_D, \Lambda_s, s\in[\tau_0, \tau_1])$ is to count the number of eigenvalues of $\mathcal{U}(\Lambda_s)$  passing $1$.  More precisely,  we  choose $\theta_0\in (0,2\pi)$ such that $e^{i\theta_0}\notin\sigma(\mathcal{U}(\Lambda_s))$ for any $s\in[\tau_0, \tau_1] $.  Denote $\nu^+(\Lambda_s)$ to be the number of total eigenvalues of $\mathcal{U}(\Lambda_s)$ in the region $\{e^{i\theta}, \theta\in (0, \theta_0) \}$.   Then we have the following lemma.
\begin{lem}\label{Mas Dir}
Under the above notation, we have
$$ \mu(\Lambda_D, \Lambda_s, s\in [\tau_0, \tau_1] )=\nu^+(\Lambda_{\tau_1})-\nu^+(\Lambda_{\tau_0}).  $$

\end{lem}
{\begin{proof}
		We use the definition of Maslov index in Remark \ref{rm:def_mas_2}.
		Since $e^{i\theta_0}\notin \sigma(\mathcal{U}(\Lambda_s))$, we see that there exist $2n$ continuous functions $\theta_j\in C([\tau_0,\tau_1],(\theta_0,\theta_0+2\pi))$ such that  $e^{i\theta_j(s)},1\leq j\leq 2n,$ are the spectrum of $\mathcal{U}(\Lambda_s)$ and
		$$
		\mu(\Lambda_D,\Lambda_s,s\in[\tau_0,\tau_1])=\sum_{j=1}^{2n} \left(E\left(\frac{\theta_j(\tau_1)}{2\pi}\right)-E\left(\frac{\theta_j(\tau_0)}{2\pi}\right)\right).
		$$
		Since $\theta_i(t)\in (\theta_0,\theta_0+2\pi)$ for all $t\in [\tau_0,\tau_1]$, we obtain
		$$
		\sum_{j=1}^{2n} E\left({\theta_j(\tau_l)\over2\pi}\right)=2n+\nu^+(\Lambda_{\tau_l}),l=0,1.
		$$
		The lemma then follows.
\end{proof}}

\begin{cor}\label{cor2.5}  Under the assumptions  \eqref{1.4}--\eqref{1.5}, we have
 $$0\leq k_\pm\leq c_0-c_\pm. $$

\end{cor}
\begin{proof}
		It suffices  to prove $0\le k_+\le c_0-c_+$.
		Note that $c_+ =\dim \left(\Lambda_s\cap \Lambda_D\right)=\dim\left(\ker(\mathcal{U}(\Lambda_s)-I_{2n})\right), s\in(0,+\epsilon]$.
Choose $\theta_0>0$ such that  $e^{it}\notin \sigma(\mathcal{U}(\Lambda_0))$ for all $t\in [-\theta_0,0)\cup(0,\theta_0]$.
Then there exists $\epsilon_0 \in(0,\epsilon)$ such that   $e^{\pm i\theta_0}\notin \sigma(\mathcal{U}(\Lambda_s))$ for all $s\in [0,\epsilon_0]$.
Therefore, the number of eigenvalues of $\mathcal{U}(\Lambda_s)$ in the region $\{e^{i\theta},\theta\in (-\theta_0,\theta_0)\} $is a constant for all $s\in [0,\epsilon_0]$ and it is exactly $c_0$.
Hence we get $\nu^+(\Lambda_{\epsilon_0})\le c_0-c_+$.
By Lemma \ref{Mas Dir} and Property III of Maslov index, we have
$k_+=\mu(\Lambda_D, \Lambda_s, s\in [0, \epsilon_0] )+\mu(\Lambda_D, \Lambda_s, s\in [\epsilon_0, \epsilon] )=\nu^+(\Lambda_{\epsilon_0})-\nu^+(\Lambda_0)=\nu^+(\Lambda_{\epsilon_0})$.
Since $0\le \nu^+(\Lambda_{\epsilon_0})\le c_0-c_+$, the conclusion then follows.
 \end{proof}

We have the following result for small perturbation of $\Lambda_{s_0}$.
\begin{lem}\label{small per}
Let $\alpha\in \text{Lag}(\mathcal{V},\omega_{2n})$ and
 \begin{equation}\label{t G-path-A-s}
  \Lambda_s=\begin{bmatrix}I_{r}&0\\0&A_s\\0&0\\0&I_{2n-r}\end{bmatrix}
  \end{equation}
  with $A_s=A_{0}+\tan (s) I_{2n-r}$. Then for any $s_0\in (-{\pi\over2},{\pi\over2})$,  $\dim(\alpha\cap \Lambda_{s})\leq r$
  for $|s-s_0|\neq0$ small enough.

\end{lem}

\begin{proof}
Without loss of generality, we assume that $s_0=0$.
 Since $\mathcal{U}(\Lambda_s)=I_{r}\oplus(A_s+iI_{2n-r})(A_s-iI_{2n-r})^{-1} $, we have
 $$
 \frac{d}{ds}\mathcal{U}(\Lambda_s)|_{s=0}\mathcal{U}(\Lambda_0)^{-1}=0_r\oplus(-2i(A_{0}^2+I_{2n-r})^{-1}) .
 $$
Define $B=\frac{d}{ds}\mathcal{U}(\Lambda_s)|_{s=0}\mathcal{U}(\Lambda_0)^{-1}$ and $C=\frac{d}{ds}(\mathcal{U}(\Lambda_s)\mathcal{U}(\alpha)^{-1}-I_{2n})|_{s=0}$.
Let $x\in \ker(\mathcal{U}(\Lambda_0)\mathcal{U}(\alpha)^{-1}-I_{2n})$,  then
$$
Cx=\frac{d}{ds}(\mathcal{U}(\Lambda_s)\mathcal{U}(\alpha)^{-1}-I_{2n})|_{s=0}(x)=\frac{d}{ds}(\mathcal{U}(\Lambda_s)\mathcal{U}(\alpha)^{-1}-I_{2n})|_{s=0}(\mathcal{U}(\Lambda_0)\mathcal{U}(\alpha)^{-1})^{-1}(x)=Bx.
$$
It follows that $ C|_{\ker(\mathcal{U}(\Lambda_0)\mathcal{U}(\alpha)^{-1}-I_{2n})}=B|_{\ker(\mathcal{U}(\Lambda_0)\mathcal{U}(\alpha)^{-1}-I_{2n})}$.

Note that $\mathbf{C}^{2n}=\ker(\mathcal{U}(\Lambda_0)\mathcal{U}(\alpha)^{-1}-I_{2n})\,\oplus\, \text{Ran}(\mathcal{U}(\Lambda_0)\mathcal{U}(\alpha)^{-1}-I_{2n})$ and it is an orthogonal decomposition.
Let $P_1$ and $ P_2$ be the orthogonal projections from $\mathbf{C}^{2n}$ to  $\text{Ran}(\mathcal{U}(\Lambda_0)\mathcal{U}(\alpha)^{-1}-I_{2n})$ and $\ker(\mathcal{U}(\Lambda_0)\mathcal{U}(\alpha)^{-1}-I_{2n})$, respectively.
Then we have $\mathcal{U}(\Lambda_s)\mathcal{U}(\alpha)^{-1}-I_{2n}=\begin{bmatrix}
A_{11}(s)& A_{12}(s)\\
A_{21}(s)& A_{22}(s)
\end{bmatrix}$ with $A_{ij}(s)=P_i(\mathcal{U}(\Lambda_s)\mathcal{U}(\alpha)^{-1}-I_{2n})|_{\text{Ran} P_j}$, $A_{12}(0)=A_{21}(0)=A_{22}(0)=0$,
and $\frac{d}{ds}A_{22}(s)|_{s=0}=P_2B|_{\text{Ran} P_2}$.

Since $iB=0_r\,\oplus \,(2(A_{0}^2+I_{2n-r})^{-1})$ is a positive semi-definite matrix,
it follows that $\text{rank}\, (P_2BP_2)\ge \dim\text{Ran} P_2-r$.
  For $|s|\neq0$ small enough, we have
  $$\dim\left(\ker(\mathcal{U}(\Lambda_s)\mathcal{U}(\alpha)^{-1}-I_{2n})\right)=\dim(\ker(A_{22}(s)-A_{21}(s)A_{11}(s)^{-1}A_{12}(s))).$$
  Since $\lim_{s\to 0}(A_{22}(s)-A_{21}(s)A_{11}(s)^{-1}A_{12}(s))/s=P_2B|_{\text{Ran} P_2}$, we have  $$\dim\left(\ker(\mathcal{U}(\Lambda_s)\mathcal{U}(\alpha)^{-1}-I_{2n})\right)\le \dim\left(\ker(P_2B|_{\text{Ran} P_2})\right)=\dim\text{Ran} P_2-\dim\text{Ran}(P_2BP_2)\le r .$$
  The lemma then follows.
\end{proof}

\section{Proof of Theorem \ref{thm: Maslov index continu}}

In this section, we provide the proof  of Theorem  \ref{thm: Maslov index continu}. To this end, we use the index form associated with $\mathcal{A}_{\Lambda_0}$ to  study the properties of eigenvalues of Sturm-Liouville problems.   Let $A$ be the Hermitian matrix determined by (\ref{eq:precise express of boundary condition}).
The  index form $I_{\Lambda_0}$ is given by
\begin{equation}\label{eq:index form}
I_{\Lambda_0}(\xi,\eta)=\int_{ 0 }^{ T }\{\langle P\dot{\xi},\dot{\eta}\rangle+\langle Q\xi,\dot{\eta}\rangle+\langle Q^{T}\dot{\xi},\eta\rangle+\langle R\xi,\eta\rangle\} dt-\langle A\begin{bmatrix} \xi(0)\\ \xi(T) \end{bmatrix},\begin{bmatrix} \eta(0)\\ \eta(T) \end{bmatrix}\rangle
 \end{equation}
 on $$H_{\Lambda_0}=\{\xi\in W^{1,2}([0,T],\mathbf{C}^n)\ | \ (\xi(0)^T,\xi(T)^T)^T\in V(\Lambda_0)\}.$$
 Obviously, $H_{\Lambda_D}= W^{1,2}_0([0,T],\mathbf{C}^n)$. For any $\xi\in E_{\Lambda_0}(0,T)$ and $\eta\in H_{\Lambda_0}$, we get by the definition of  $\mathcal{A}_{\Lambda_0}$  and integration by parts that
 \begin{equation*}
\begin{aligned}
 \langle \mathcal{A}_{\Lambda_0}\xi,\eta\rangle_{L^2}&=\int_{ 0 }^{ T }\langle-( P\dot{\xi}+Q\xi)^{\cdot}+Q^{T}\dot{\xi}+R\xi,\eta\rangle dt\\
 &=\int_{ 0 }^{ T }\{\langle P\dot{\xi},\dot{\eta}\rangle+\langle Q\xi,\dot{\eta}\rangle+\langle Q^{T}\dot{\xi},\eta\rangle +\langle R\xi,\eta\rangle\} dt -\langle\begin{bmatrix} -y(0)\\y(T) \end{bmatrix},\begin{bmatrix} \eta(0)\\ \eta(T) \end{bmatrix}\rangle\\
 &=\int_{ 0 }^{ T }\{\langle P\dot{\xi},\dot{\eta}\rangle+\langle Q\xi,\dot{\eta}\rangle+\langle Q^{T}\dot{\xi},\eta\rangle+\langle R\xi,\eta\rangle\} dt-\langle A\begin{bmatrix} \xi(0)\\ \xi(T) \end{bmatrix},\begin{bmatrix} \eta(0)\\ \eta(T) \end{bmatrix}\rangle\\
&=I_{\Lambda_0}(\xi,\eta),
\end{aligned}
\end{equation*}
where $y(t)=P(t)\dot{\xi}(t)+Q(t)\xi(t)$. Here  the third equality holds due to  the boundary condition $\Lambda_0$ in the frame \eqref{eq:precise express of boundary condition}.

  Let $\lambda_j(\Lambda_0)$ be the $j$-th eigenvalue of $\mathcal{A}_{\Lambda_0}$.
 From the minmax property of eigenvalues (see \cite{CH,RS}), we have
 \bea \lambda_j(\Lambda_0)=\sup_{E_{j-1}}\ \inf_{\xi\in E^\perp_{j-1},\,  \xi\neq0} \ \frac{I_{\Lambda_0}(\xi,\xi)}{\|\xi\|^2_{L^2}}, \label{eigen vara} \eea
 where $E_{j-1}$ is any $j-1$ dimensional closed subspace of $H_{\Lambda_0}$.

Next, we will decompose $H_{\Lambda_0}$ into two subspaces. Let $\dim V(\Lambda_0)=k_0$ and  ${\tilde\lambda}_{k_0}(\Lambda_0)\leq\cdots\leq\tilde{ \lambda}_1(\Lambda_0)$ be all the eigenvalues of $A$ with $e_i\in V(\Lambda_0)$, $1\leq i\leq k_0$, to be the correspondingly normalized eigenvectors. For any $1\leq i\leq k_0$, we can construct a linear function $\xi_i$ with  $\xi_i(t)=\xi_i(0)+\frac{t}{T}(\xi_i(T)-\xi_i(0))$, $t\in[0,T]$,  such that $(\xi_i(0)^T,\xi_i(T)^T)^T=e_i$.   For any $1\leq l\leq k_0$, we set
$$\mathcal{X}_l(\Lambda_0)=\text{span}\{\xi_i: 1\leq i\leq l\},$$ which is a $l$ dimensional subspace of  $H_{\Lambda_0}$. We define a new  norm of $\xi\in\mathcal{X}_{k_0}(\Lambda_0)$ by
 $$ \|\xi\|_2:=\sqrt{\|\xi(0)\|_{\mathbf{C}^n}^2+\|\xi(T)\|_{\mathbf{C}^n}^2}.$$
 It is well-defined since $\mathcal{X}_{k_0}(\Lambda_0)$ consists of linear functions.
 Since $\mathcal{X}_l(\Lambda_0)$ is finite dimensional,  we shall show that this  norm  is equivalent to the $L^2$ norm and the $W^{1,2}$ norm.   More precisely, we have the following lemma.

 \begin{lem}\label{equinorm}  There exist $c^\pm_1, c^\pm_2>0$, which depend only on $T$, such that
\bea \label{norm-inequalities}
 c_1^- \|\xi\|_2\leq  \|\xi\|_{L^2}\leq  c_1^+ \|\xi\|_2, \quad   c_2^- \|\xi\|_2\leq  \|\xi\|_{W^{1,2}}\leq  c_2^+ \|\xi\|_2,\quad \forall \;\xi\in\mathcal{X}_{k_0}(\Lambda_0).
  \eea
\end{lem}
\begin{proof}  	Let $V=\{\xi:\xi(t)=a+{t\over T}(b-a),t\in[0,T],(0,0,a^T,b^T)^T\in \Lambda_N\}$.
	Then the dimension of $V\cong \Lambda_N $ is $2n$.
	So there exist $c^\pm_1, c^\pm_2>0$ such that
$$
	c_1^- \|\xi\|_2\leq  \|\xi\|_{L^2}\leq  c_1^+ \|\xi\|_2, \quad   c_2^- \|\xi\|_2\leq  \|\xi\|_{W^{1,2}}\leq  c_2^+ \|\xi\|_2,\quad \forall \;\xi\in V.
$$
 Now \eqref{norm-inequalities} follows from the fact that $\mathcal{X}_{k_0}(\Lambda_0)\subset V $ for any $\Lambda_0$.  On the other hand,  we can also prove   \eqref{norm-inequalities}  by direct computation.   More precisely,
$$
\begin{aligned}
\left({T\over6}\right)^{1\over2}\|\xi\|_2\leq \|\xi\|_{L^2}\leq\left({T\over2}\right)^{1\over2}\|\xi\|_2,
\end{aligned}
$$
and
\begin{equation*}
\begin{aligned}
\left({T\over6}\right)^{1\over2}\|\xi\|_2\leq\|\xi\|_{W^{1,2}}\leq\left({T^2+4\over2 T} \right)^{1\over2}\|\xi\|_2,
\end{aligned}
\end{equation*}
which  gives the exact values of $c^\pm_1, c^\pm_2$.
  \end{proof}

Denote $H_0=H_{\Lambda_D}$ for convenience.  Then we get the decomposition of $H_{\Lambda_0}$.
 \begin{lem}\label{lem:decompo}
$$ H_{\Lambda_0}=H_0\oplus \mathcal{X}_{k_0}(\Lambda_0). $$
\end{lem}
\begin{proof}  For any $x\in H_{\Lambda_0} $, we have $x\in W^{1,2}([0,T], \mathbf{C}^{n}   )$ and $(x(0)^T,x(T)^T)^T\in V(\Lambda_0)$. We choose $\tilde{x}\in  \mathcal{X}_{k_0}(\Lambda_0) $ such that $(\tilde{x}(0), \tilde{x}(T))=(x(0),x(T))$. Then $x-\tilde{x}\in H_0$.  Since $\mathcal{X}_{k_0}(\Lambda_0)$ consists of  linear functions, we get $\mathcal{X}_{k_0}(\Lambda_0) \cap H_0=\{0\}$.
 \end{proof}

Now we are ready to  give the relationship of $\lambda_j$ and $\tilde \lambda_j$.

 \begin{prop}\label{prop:low bound}  Let $\mathcal{S}\subset \text{Lag}(\mathcal{V} , \omega_{2n}) $ and $c\in\mathbf{R}$ such that  $\tilde{\lambda}_j(\Lambda)\leq c$ and $\dim V(\Lambda)=k_0$ for all $\Lambda\in \mathcal{S}$. Then  $\lambda_j$ has a uniformly lower bound on $\mathcal{S}$.
\end{prop}
\begin{proof}
For any $\xi\in H_0\oplus (\mathcal{X}_{k_0}(\Lambda)\ominus\mathcal{X}_{j-1}(\Lambda))$ and any $\varepsilon>0$, there exists $C_\varepsilon>0$, independent of $\Lambda\in\mathcal{S}$, such that
$$\|\xi\|_{C}^2\leq \varepsilon\|\dot{\xi}\|_{L^2}^2+C_\varepsilon\|\xi\|_{L^2}^2,$$
and
$$\langle A\begin{bmatrix} \xi(0)\\ \xi(T) \end{bmatrix},\begin{bmatrix} \xi(0)\\ \xi(T) \end{bmatrix}\rangle\leq \tilde \lambda_j(\Lambda)\|\xi\|_2^2\leq 2c\|\xi\|_{C}^2\leq 2c\varepsilon\|\dot{\xi}\|_{L^2}^2+2cC_\varepsilon\|\xi\|_{L^2}^2.$$
Then we get by (\ref{eq:index form}) that there exists $c_1>0$, $c_2<0$ and $\tilde c_1\in\mathbf{R}$, independent of $\Lambda\in\mathcal{S}$, such that
\begin{equation*}
\begin{aligned}
I_{\Lambda}(\xi,\xi)\geq &c_1\|\dot{\xi}\|_{L^2}^2+c_2\|\xi\|_{L^2}\|\dot{\xi}\|_{L^2}+\tilde c_1\|\xi\|_{L^2}^2- 2c\varepsilon\|\dot{\xi}\|_{L^2}^2-2cC_\varepsilon\|\xi\|_{L^2}^2\\
\geq& c_1\|\dot{\xi}\|_{L^2}^2+c_2\varepsilon\|\dot{\xi}\|_{L^2}^2+{c_2\over\varepsilon}\|\xi\|_{L^2}^2+\tilde c_1\|\xi\|_{L^2}^2- 2c\varepsilon\|\dot{\xi}\|_{L^2}^2-2cC_\varepsilon\|\xi\|_{L^2}^2\\
=&(c_1+c_2\varepsilon-2c\varepsilon)\|\dot{\xi}\|_{L^2}^2+\left({c_2\over\varepsilon}+\tilde c_1-2cC_\varepsilon\right)\|\xi\|_{L^2}^2
\geq\left({c_2\over\varepsilon}+\tilde c_1-2cC_\varepsilon\right)\|\xi\|_{L^2}^2
\end{aligned}
\end{equation*}
for  $\xi\in H_0\oplus (\mathcal{X}_{k_0}(\Lambda)\ominus\mathcal{X}_{j-1}(\Lambda))$, where $\varepsilon>0$ is small enough such that
$c_1+c_2\varepsilon-2c\varepsilon>0$.
By Lemma \ref{lem:decompo},
$H_0\oplus (\mathcal{X}_{k_0}(\Lambda)\ominus\mathcal{X}_{j-1}(\Lambda))  $ is a closed subspace of $H_{\Lambda}$ with codimension $j-1$.
Then we get by (\ref{eigen vara}) that
$\lambda_j(\Lambda)\geq {c_2\over\varepsilon}+\tilde c_1-2cC_\varepsilon$ for all $\Lambda\in\mathcal{S}$.
 \end{proof}

  \begin{prop}\label{prop: eigen upper bound}  Let $\mathcal{S}\subset \text{Lag}(\mathcal{V} , \omega_{2n}) $ and $\dim V(\Lambda)=k_0$ for all $\Lambda\in \mathcal{S}$. Then
   for any $1\leq j\leq k_0$,  there exist   $c_3>0$ and $c_4\in\mathbf{R}$, which are independent of $\Lambda$,
  such that
  \begin{equation*} \lambda_j(\Lambda)\leq -c_3 \tilde{\lambda}_j(\Lambda)+c_4,\;\; \Lambda\in \mathcal{S}. \end{equation*}
\end{prop}
\begin{proof}
For any $\xi\in \mathcal{X}_j(\Lambda)$,
   we have  by
   \eqref{eq:index form} and Lemma \ref{equinorm} that
     \begin{equation*}
 \begin{aligned}
 I_\Lambda(\xi,\xi) =& \int_{ 0 }^{ T }\{\langle P\dot{\xi},\dot{\xi}\rangle+\langle Q\xi,\dot{\xi}\rangle+\langle Q^{T}\dot{\xi},\xi\rangle+\langle R\xi,\xi\rangle\} dt-\langle A\begin{bmatrix} \xi(0)\\ \xi(T) \end{bmatrix},\begin{bmatrix} \xi(0)\\ \xi(T) \end{bmatrix}\rangle  \nonumber \\
   \leq & c\|\xi\|_2^2-\tilde{\lambda}_j(\Lambda) \|\xi\|_2^2\leq c_4\|\xi\|_{L^2}^2-c_3\tilde\lambda_j(\Lambda)\|\xi\|_{L^2}^2,
    \end{aligned}
    \end{equation*}
 where $c$ and $c_4$ depend only on $P, Q, R$ and $T$, $c_3=(c_1^+)^{-2}>0$ if $\tilde \lambda_j(\Lambda)>0$, and $c_3=(c_1^-)^{-2}>0$ if $\tilde \lambda_j(\Lambda)<0$. Since $\dim \mathcal{X}_j(\Lambda)=j$,  we have $ \mathcal{X}_j(\Lambda)\cap E^\perp_{j-1}\neq \{0\} $ for any fixed $j-1$ dimensional subspace  $E_{j-1}$ of $H_{\Lambda}$, and thus
  \begin{equation*} \inf_{\xi\in E^\perp_{j-1},\,  \xi\neq0} \ \frac{I_\Lambda(\xi,\xi)}{\|\xi\|_{L^2}^2}\leq -c_3 \tilde{\lambda}_j(\Lambda)+c_4.   \end{equation*}
The proof is complete  by \eqref{eigen vara}.
 \end{proof}

 Then we give some criteria for the continuity of $\lambda_j$.

  \begin{lem}\label{con: continuous}

  (1) Let $\mathcal{S}\subset \text{Lag}(\mathcal{V},\omega_{2n})$ and $\lambda_1$ be uniformly  bounded from below on $\mathcal{S}$. Then $\lambda_j$ is continuous on $\mathcal{S}$ for all $j\geq 1$.

  (2)  Let $\Lambda_s, s\in[0,\epsilon],$  be a continuous path in $\text{Lag}(\mathcal{V},\omega_{2n})$.  If $\lim_{s\to0^+}\lambda_j(\Lambda_s)=-\infty$ for all $1\leq j\leq j_0$, and $\lambda_{j_0+1}(\Lambda_s)$, $s\in (0,\epsilon]$, have a uniformly lower bound, then we have
  \begin{equation*}
  \lim_{s\to0^+}\lambda_j(\Lambda_s)=\lambda_{j-j_0}(\Lambda_0)
  \end{equation*}
for all $j>j_0$.
\end{lem}
\begin{proof}
We first prove (1).
Let $r_1<\inf_{\Lambda\in\mathcal{S}}\lambda_1(\Lambda)$, $\Lambda_0\in\mathcal{S}$, and $j_0\geq1$   such that $\lambda_{j_0+1}(\Lambda_{0})> \lambda_{j_0}(\Lambda_{0})$.
Choose $r_2\in(\lambda_{j_0}(\Lambda_0),\lambda_{j_0+1}(\Lambda_0))$. It follows from
 Theorem 3.16 in \cite{Kato1984} that there exists a neighborhood $\mathcal{S}_0\subset\mathcal{S}$ of $\Lambda_0$  such that
 there are exactly $j_0$ eigenvalues (counting multiplicity)  of $\mathcal{A}_{\Lambda}$ with $\Lambda\in\mathcal{S}_0$ in $(r_1,r_2)$.
 Since $\lambda_1(\Lambda)>r_1$ for $\Lambda\in\mathcal{S}_0$, the above $j_0$ eigenvalues are exactly $\lambda_j(\Lambda), 1\leq j\leq j_0$.
    Let $\epsilon>0$ be small enough such that the intervals with radius $\epsilon>0$  centred at  the non-equal ones  of $\lambda_j(\Lambda_{0}), 1\leq j\leq j_0$, are contained in $(r_1,r_2)$. By  Theorem 3.16 in \cite{Kato1984} again, there exists $\mathcal{S}_1\subset\mathcal{S}_0$  such that $|\lambda_j(\Lambda)-\lambda_j(\Lambda_{0})|<\epsilon$ for any $\Lambda\in\mathcal{S}_1$. Therefore, (1) holds.

(2) can be shown by a similar method, and thus we omit the details.
 \end{proof}

Next, we study the asymptotic behavior of $\tilde\lambda_j$.

 \begin{lem} \label{asymptotic behavior tilde}
  Assume that $\Lambda_s, s\in[-\epsilon,\epsilon],$ satisfy \eqref{1.4} and \eqref{1.5}. Then we have
 \begin{equation}\label{no-tilde-and-tilde1}  \lim_{s\to 0^-} \tilde{\lambda}_j(\Lambda_s)=+\infty \quad  for  \quad  1\leq j\leq k_- ,
 \end{equation}
 and
 there exists $M^->0$ such that  $\tilde{\lambda}_j(\Lambda_s)\leq M^-$  on $s\in[-\epsilon,0)$ for $j> k_-$. Similarly,
\begin{equation*} \lim_{s\to 0^+} \tilde{\lambda}_j(\Lambda_s)=+\infty \quad  for  \quad  1\leq j\leq k_+ ,
\end{equation*}
 and
 there exists $M^+>0$ such that  $\tilde{\lambda}_j(\Lambda_s)\leq M^+$  on  $s\in(0, \epsilon]$ for $j> k_+$.
 \end{lem}

  \begin{proof}  We only prove the first conclusion, since others can be shown similarly.
		For any $\beta\in(0,\pi),$ there exists $\alpha\in(0,
		\beta)$  such that $S_\alpha\cap \sigma(\mathcal{U}(\Lambda_0))=\emptyset$, where $S_\alpha=\{e^{i\theta}|\theta\in (0,\alpha]\}$.
		So there exists $r\in(0,\epsilon)$ such that $e^{i\alpha }\notin \sigma(\mathcal{U}(\Lambda_s)), -r<s<0$.	
	It follows from Lemma \ref{Mas Dir} that  $\#(S_\alpha\cap \sigma(\mathcal{U}(\Lambda_s)))= k_-$.
Note that $\mathcal{U}(\Lambda_s)=\begin{pmatrix}I_{c_-}&0\\
0&(A_s+iI_{2n-c_-})(A_s-iI_{2n-c_-})^{-1}
\end{pmatrix}$, and
 thus there are exactly $k_-$ eigenvalues, denoted by $\tilde{ \lambda}_j(\Lambda_s),1\le j\le k_-$, of $A_s$ such that $(\tilde\lambda_j(\Lambda_s)+i)(\tilde\lambda_j(\Lambda_s)-i)^{-1}\in S_\alpha $ with $s\in (-r,0)$.
This implies $\tilde{\lambda}_j(\Lambda_s)>i(e^{i\alpha}+1)/(e^{i\alpha}-1)=\cot(\alpha/2)>\cot(\beta/2)$.
By the arbitrary choice of $\beta$,  we have

\begin{equation*} \lim_{s\to 0^-} \tilde{\lambda}_j(\Lambda_s)=+\infty \quad  \text{for}  \quad  1\leq j\leq k_-.
\end{equation*}
Fix any $\beta_0\in(0,\pi).$ Since $(\tilde\lambda_j(\Lambda_s)+i)(\tilde\lambda_j(\Lambda_s)-i)^{-1}\notin S_{\alpha_0}$ for all $-r_0<s<0$ and all $j>k_-$, we infer that $\tilde\lambda_j(\Lambda_s)<\cot(\alpha_0/2)$.
\end{proof}

Then we study the asymptotic behavior of $\lambda_j$  using that of $\tilde\lambda_j$.

\begin{prop}\label{prop: discontinuous}
  Assume that $\Lambda_s, s\in[-\epsilon,\epsilon],$ satisfy \eqref{1.4} and \eqref{1.5}. Then for any $j\geq1$,
 \bea  \lim_{s\to0^-}\lambda_j(\Lambda_s)=\lambda_{j-k_-}(\Lambda_0),
 \label{prop dis for neg}    \eea
    and
 \bea  \lim_{s\to0^+}\lambda_j(\Lambda_s)=\lambda_{j-k_+}(\Lambda_0).     \label{prop dis for pos}    \eea
     \end{prop}

\begin{proof}
We only prove (\ref{prop dis for neg}), and (\ref{prop dis for pos}) can be shown in a similar way.

Let $1\leq j\leq k_-$. Then by Proposition \ref{prop: eigen upper bound},
  $\lambda_j(\Lambda_s)\leq -c_3 \tilde{\lambda}_j(\Lambda_s)+c_4,$ where $c_3>0$,   $s\in[-\epsilon,0)$.
Thanks to (\ref{no-tilde-and-tilde1}), we have $\lim_{s\to0^-}\lambda_j(\Lambda_s)=-\infty$.

Let $j> k_-$. By  Lemma  \ref{asymptotic behavior tilde}, there exists
$M^->0$ such that  $\tilde{\lambda}_j(\Lambda_s)\leq M^-$ on $s\in[-\epsilon,0)$ for $j> k_-$. In view of Proposition \ref{prop:low bound}, we have
$\lambda_j(\Lambda_s)$,  $s\in[-\epsilon,0)$, have a uniformly lower bound for any $j> k_-$. Then it follows from (2) of Lemma \ref{con: continuous} that
$\lim_{s\to0^-}\lambda_{j}(\Lambda_s)=\lambda_{j-k_-}(\Lambda_0)$.
 \end{proof}

Now we are in a position to show Theorem \ref{thm: Maslov index continu}.


 \begin{proof}[Proof of Theorem \ref{thm: Maslov index continu}]
 We first prove that $\lambda_j$, $j\geq1$, are all continuous on $\{\Lambda_s:s\in[-\epsilon,0)\}$.
Since $\mu(\Lambda_D, \Lambda_s, s\in[-\epsilon,s_0])=0$ for any $s_0\in(-\epsilon,0)$, we have by Lemma \ref{asymptotic behavior tilde} that
$\tilde \lambda_1(\Lambda_s)<M^-$ for all $s\in[-\epsilon,s_0)$. Thanks to Proposition \ref{prop:low bound}, we get  that $\lambda_1(\Lambda_s)$,  $s\in[-\epsilon,s_0)$, have a uniformly lower bound. Then by (1) of Lemma \ref{con: continuous} and the arbitrary choice of $s_0\in(-\epsilon,0)$, we obtain the result. The continuity of  $\lambda_j$ on  $\{\Lambda_s:s\in(0,\epsilon]\}$ can be shown similarly.

Please note that  (\ref{jump12}) is obtained by Proposition \ref{prop: discontinuous}.
 \end{proof}

\section{Proof of Theorem \ref{thm: eigen range}}

In this section, we give the proof of Theorem \ref{thm: eigen range}.

 \begin{proof}[Proof of Theorem \ref{thm: eigen range}] The proof is complete by Propositions \ref{prop: range with out endpoints consideration}, \ref{left endpoint} and \ref{right endpoint}.
\end{proof}

\begin{proposition}\label{prop: range with out endpoints consideration} Fix any $j\geq1$ and $0\leq r\leq 2n$. Then  $$(\lambda_{j-(2n-r)}(\Lambda_D), \lambda_j(\Lambda_D))\subset \lambda_j(\Sigma_{r})\subset[\lambda_{j-(2n-r)}(\Lambda_D),  \lambda_j(\Lambda_D)].$$
\end{proposition}

 \begin{proof}
 Let $A_0$ be the Hermitian matrix in the Lagrangian frame of $\Lambda_0\in\Sigma_{r}$. We define $\Lambda_s\in \text{Lag} (\mathcal{V}, \omega_{2n})$ by
(\ref{t G-path-A-s}),
where $A_s=A_0+\tan (s) I_{2n-r}$ for $s\in(-\pi/2, \pi/2)$. Noting that $\Lambda_{\pm\pi/2}:=\lim_{s\to\pm{\pi\over2}}\Lambda_s=\Lambda_D$,
 $\{\Lambda_s,s\in[-\pi/2, \pi/2]\}$ is a continuous loop.  It is obvious that
 \begin{equation*}  \dim \Lambda_s\cap \Lambda_D=r \quad \text{for} \quad s\in (-\pi/2, \pi/2).
 \end{equation*}
  Direct computation gives
  \begin{equation}\label{maslov-index-pos-neg} \mu(\Lambda_D, \Lambda_s, s\in [0, \pi/2])=-(2n-r), \quad    \mu(\Lambda_D, \Lambda_s, s\in [-\pi/2, 0])=0.   \end{equation}
 Recall that $I_{\Lambda_s}$ is the corresponding index form, we have
  \begin{equation*} I_{\Lambda_{s_1}}\geq I_{\Lambda_{s_2}}, \quad \text{if} \quad s_1\leq s_2.  \end{equation*}
By (\ref{eigen vara}) we get
 \begin{equation} \lambda_j(\Lambda_{s_1})\geq \lambda_j(\Lambda_{s_2}),   \quad \text{if} \quad s_1\leq s_2.  \label{inequality:lambdaj}
  \end{equation}
  Letting $s_2=0$ and $s_1\to(-\pi/2)^+$ in (\ref{inequality:lambdaj}),  we get by (\ref{maslov-index-pos-neg}) and Theorem  \ref{thm: Maslov index continu} that
   \begin{equation}\label{1.2.1-first part} [\lambda_j(\Lambda_{0}),  \lambda_j(\Lambda_{D}))\subset \lambda_j(\Sigma_r)\;\; \textrm {and}\;\; \lambda_j(\Lambda_{0})\leq\lambda_j(\Lambda_{D}). \end{equation}
  On the other hand,  letting $s_1=0$ and $s_2\to (\pi/2)^-$ in (\ref{inequality:lambdaj}), we infer again from (\ref{maslov-index-pos-neg}) and Theorem \ref{thm: Maslov index continu} that
 \begin{equation}\label{1.2.1-second part}(\lambda_{j-(2n-r)}(\Lambda_{D}),  \lambda_j(\Lambda_{0})]\subset \lambda_j(\Sigma_r)
 \;\; \textrm {and}\;\; \lambda_j(\Lambda_{0})\geq\lambda_{j-(2n-r)}(\Lambda_{D}). \end{equation}
  Then the conclusion  is proved by (\ref{1.2.1-first part}) and (\ref{1.2.1-second part}).
 \end{proof}

Next, we study the left endpoint of the range $\lambda_j(\Sigma_r)$.
 \begin{proposition}\label{left endpoint}
Fix any $0\leq r\leq 2n$ and $j>2n-r$.
Let
$$\lambda_{j-(2n-r)-b_1}(\Lambda_D)=\cdots=\lambda_{j-(2n-r)}(\Lambda_D)=\cdots=\lambda_{j-(2n-r)+b_2}(\Lambda_D)$$
with multiplicity to be $b_1+b_2+1$,  where $b_i\geq 0$, $i=1,2$.
Then we have two cases.

(1) If $r\leq b_1$, then for any $\Lambda\in\Sigma_r$,
 $$\lambda_j(\Lambda)>\lambda_{j-(2n-r)}(\Lambda_D).$$

(2) If $r>b_1$, then
 $$\min \lambda_j(\Sigma_r)=\lambda_{j-(2n-r)}(\Lambda_D).$$
 \end{proposition}
\begin{proof}
  Firstly, we prove (1). Suppose that there exists $\Lambda_0\in \Sigma_r$ such that $\lambda_j(\Lambda_0)=\lambda_{j-(2n-r)}(\Lambda_D)$.
 Since $ 0\leq r\leq b_1$, $\lambda_{j-(2n-r)-b_1}(\Lambda_D)= \lambda_{j-2n}(\Lambda_D)=\lambda_{j-(2n-r)}(\Lambda_D)$. By Proposition \ref{prop: range with out endpoints consideration},  we have
 $\lambda_{j-2n}(\Lambda_D)\leq\lambda_{j-r}(\Lambda_0)$. Thus
 \begin{equation}\label{Lambda0r+1}\lambda_{j-2n}(\Lambda_D)=\lambda_{j-r}(\Lambda_0)= \cdots = \lambda_{j}(\Lambda_0)=:\lambda.\end{equation}
  Let $\Lambda_s$  be defined by (\ref{t G-path-A-s}),
where $A_s=A_0+\tan (s) I_{2n-r}$ for $s\geq0$. Thanks to  Proposition \ref{prop: range with out endpoints consideration} and the fact that   $I_{\Lambda_s}\leq I_{\Lambda_0}$ for  $s>0$,
we have $\lambda_{i-(2n-r)}(\Lambda_D)\leq\lambda_i(\Lambda_s)\leq \lambda_i({\Lambda_0})$ for all $i\geq1$.
By (\ref{Lambda0r+1}) we get that for $s\geq0$,
\begin{equation*} \lambda_{j-2n}(\Lambda_D)=\lambda_{j-r}(\Lambda_s)=\cdots=\lambda_{j}(\Lambda_s)=\lambda.
 \end{equation*}
 Then $\lambda$ is an eigenvalue of $\mathcal{A}_{\Lambda_s}$ with multiplicity to be at least $r+1$ and thus
\begin{equation}
\label{contradiction}\dim (Gr(\gamma_\lambda(T))\cap\Lambda_s)\geq r+1,\;\;\;\; s\geq0.\end{equation}
On the other hand,
we get by Lemma \ref{small per} that  for $s>0$ small enough,
\begin{equation*}
\dim (Gr(\gamma_\lambda(T))\cap\Lambda_s)\leq r, \end{equation*}
which is a contradiction to (\ref{contradiction}).

Next, we show that (2) holds. Let $l_1=b_1+b_2+1$ and $\alpha_0=Gr(\gamma_\lambda(T))\cap\Lambda_D$ with $\lambda:=\lambda_{j-(2n-r)}(\Lambda_D)$ for convenience. Then $\dim \alpha_0=l_1$. We divide the proof into two cases.

 Case 1: $r\geq l_1$.  Let $\tilde\Lambda_1=\alpha_0\oplus V \oplus W_{s_0}$, where $V\subset\Lambda_D\ominus \alpha_0$, $\dim V=r-l_1$ and
\begin{equation}\label{Ws0}
W_{s_0}=\begin{bmatrix}0_{r}&0\\0&({\tan(s_0)}+1)I_{2n-r}\\0&0\\0&I_{2n-r}\end{bmatrix}
\end{equation}
for ${\pi\over2}-s_0>0$ small enough. Then $\tilde\Lambda_1\in \Sigma_r$. By  Lemma \ref{small per} and the construction of $\tilde\Lambda_1$, $\dim(\tilde\Lambda_1\cap Gr(\gamma_\lambda(T)))=l_1$.
Let $\epsilon>0$ be small enough such that $\lambda$ is the only eigenvalue of $\mathcal{A}_{\Lambda_D}$ in $ [\lambda-\epsilon,\lambda+\epsilon]$. By Theorem 3.16 in \cite{Kato1984}, there are exactly $l_1$ eigenvalues (counting multiplicity) of $\mathcal{A}_{\tilde\Lambda_1}$ in  $ [\lambda-\epsilon,\lambda+\epsilon]$. They are    $\lambda_{j-b_1}(\tilde\Lambda_1)\leq\cdots\leq\lambda_{j}(\tilde\Lambda_1)\leq\cdots\leq\lambda_{j+b_2}(\tilde\Lambda_1)$ by Theorem  \ref{thm: Maslov index continu}.
 Since $\dim(\tilde\Lambda_1\cap Gr(\gamma_\lambda(T)))=l_1$, we have $\lambda_{j-b_1}(\tilde\Lambda_1)=\cdots=\lambda_{j}(\tilde\Lambda_1)=\cdots=\lambda_{j+b_2}(\tilde\Lambda_1)=\lambda$.
Therefore, $\lambda_{j}(\tilde\Lambda_1)=\lambda_{j-(2n-r)}(\Lambda_D)$.

 Case 2: $b_1<r< l_1$. Let $\tilde\Lambda_2=U\oplus W_{s_0}$, where $U\subset\alpha_0$, $\dim U=r$ and $W_{s_0}$ is given in (\ref{Ws0}) for ${\pi\over2}-s_0>0$ small enough. Then  $\tilde\Lambda_2\in \Sigma_r$. By Lemma \ref{small per}, $s_0$ can be chosen such that
 $\dim(\tilde\Lambda_2\cap Gr(\gamma_\lambda(T)))= r$.
 Similar to Case 1,   $\lambda_{j-b_1}(\tilde\Lambda_2)\leq\cdots\leq\lambda_{j}(\tilde\Lambda_2)\leq\cdots\leq\lambda_{j+b_2}(\tilde\Lambda_2)$  are all the   eigenvalues  of $\mathcal{A}_{\tilde\Lambda_2}$ in  $ [\lambda-\epsilon,\lambda+\epsilon]$.
 Since $\dim(\tilde\Lambda_2\cap Gr(\gamma_\lambda(T)))= r$ and  $\lambda_{i-(2n-r)}(\Lambda_D)\leq \lambda_i(\tilde\Lambda_2)$ for  $j-b_1\leq i\leq j+b_2$, we have $\lambda_{j-b_1}(\tilde\Lambda_2)=\cdots=\lambda_{j}(\tilde\Lambda_2)=\cdots=\lambda_{j+(r-b_1-1)}(\tilde\Lambda_2)=\lambda<\lambda_{j+(r-b_1)}(\tilde\Lambda_2)$.
Therefore, $\lambda_{j}(\tilde\Lambda_2)=\lambda_{j-(2n-r)}(\Lambda_D)$.
\end{proof}


Finally, we study the right endpoint of the range $\lambda_j(\Sigma_r)$.
\begin{proposition}\label{right endpoint}
Fix any $j\geq1$ and $0\leq r\leq 2n$.
Let
$$\lambda_{j-c_1}(\Lambda_D)=\cdots=\lambda_{j}(\Lambda_D)=\cdots=\lambda_{j+c_2}(\Lambda_D)$$
with multiplicity to be $c_1+c_2+1$,  where $c_i\geq 0$, $i=1,2$.
Then we have two cases.

(1) If $r\leq c_2$, then for any $\Lambda\in\Sigma_r$,
 $$\lambda_j(\Lambda)<\lambda_{j}(\Lambda_D).$$

(2) If $r>c_2$, then
 $$\max \lambda_j(\Sigma_r)=\lambda_{j}(\Lambda_D).$$
 \end{proposition}
 \begin{proof} The method is similar as Proposition \ref{left endpoint} and we give the proof here for completeness.
 We first prove (1).
 Suppose that there exists $\Lambda_0\in \Sigma_r$ such that $\lambda_j(\Lambda_0)=\lambda_{j}(\Lambda_D)$.
 Since $ 0\leq r\leq c_2$, $\lambda_{j}(\Lambda_D)= \lambda_{j+r}(\Lambda_D)=\lambda_{j+c_2}(\Lambda_D)$. By Proposition \ref{prop: range with out endpoints consideration},  we have
 $ \lambda_{j+r}(\Lambda_0)\leq\lambda_{j+r}(\Lambda_D)$. Thus
 \begin{equation}\label{Lambda0r+1right}\lambda_{j}(\Lambda_0)=\cdots= \lambda_{j+r}(\Lambda_0) = \lambda_{j+r}(\Lambda_D)=:\lambda.\end{equation}
  Let
 $\Lambda_s$ be given by (\ref{t G-path-A-s}),
where $A_s=A_0+\tan (s) I_{2n-r}$ for $s\leq0$. Thanks to  Proposition \ref{prop: range with out endpoints consideration} and the fact that   $I_{\Lambda_s}\geq I_{\Lambda_0}$ for  $s<0$,
we have $\lambda_i({\Lambda_0})\leq\lambda_i(\Lambda_s)\leq \lambda_{i}(\Lambda_D)$ for all $i\geq1$.
By (\ref{Lambda0r+1right}) we get that for $s\leq0$,
\begin{equation*} \lambda_{j}(\Lambda_s)=\cdots=\lambda_{j+r}(\Lambda_s)=\lambda_{j+r}(\Lambda_D)=\lambda.
 \end{equation*}
 Then
\begin{equation}
\label{contradiction-right}\dim (Gr(\gamma_\lambda(T))\cap\Lambda_s)\geq r+1,\;\;\;\; s\leq0.\end{equation}
However,
we get by Lemma \ref{small per} that  for $s<0$ small enough,
\begin{equation*}
\dim (Gr(\gamma_\lambda(T))\cap\Lambda_s)\leq r, \end{equation*}
which contradicts (\ref{contradiction-right}).

Next, we prove (2). Let $l_2=c_1+c_2+1$ and $\beta_0=Gr(\gamma_\lambda(T))\cap\Lambda_D$ with $\lambda:=\lambda_{j}(\Lambda_D)$. Then $\dim \beta_0=l_2$. We divide the proof into two cases.

 Case 1: $r\geq l_2$.  Let $\hat\Lambda_1=\beta_0\oplus  V\oplus W_{s_0}$, where $V\subset\Lambda_D\ominus \beta_0$, $\dim V=r-l_2$ and
$W_{s_0}$ is given by (\ref{Ws0})
for $s_0+{\pi\over2}>0$ small enough. Then $\hat\Lambda_1\in \Sigma_r$. By  Lemma \ref{small per} and the construction of $\hat\Lambda_1$, $\dim(\hat\Lambda_1\cap Gr(\gamma_\lambda(T)))=l_2$.
Let $\epsilon>0$ be small enough such that $\lambda$ is the only eigenvalue of $\mathcal{A}_{\Lambda_D}$ in $ [\lambda-\epsilon,\lambda+\epsilon]$. By Theorem 3.16 in \cite{Kato1984}, there are exactly $l_2$ eigenvalues of $\mathcal{A}_{\hat\Lambda_1}$ in  $ [\lambda-\epsilon,\lambda+\epsilon]$. They are
  $\lambda_{j-c_1}(\hat\Lambda_1)=\cdots=\lambda_{j}(\hat\Lambda_1)=\cdots=\lambda_{j+c_2}(\hat\Lambda_1)=\lambda$
by Theorem  \ref{thm: Maslov index continu}
 and the fact that $\dim(\hat\Lambda_1\cap Gr(\gamma_\lambda(T)))=l_2$.
Therefore, $\lambda_{j}(\hat\Lambda_1)=\lambda_{j}(\Lambda_D)$.

 Case 2: $c_2<r< l_2$. Let $\hat\Lambda_2=U\oplus W_{s_0}$, where $U\subset\beta_0$ and $\dim U=r$ and $W_{s_0}$ is given in (\ref{Ws0}) for $s_0+{\pi\over2}>0$ small enough. Then  $\hat\Lambda_2\in \Sigma_r$. By Lemma \ref{small per}, $s_0$ can be chosen such that
 $\dim(\hat\Lambda_2\cap Gr(\gamma_\lambda(T)))= r$. Similar to Case 1,
  $\lambda_{j-c_1}(\hat\Lambda_2)\leq\cdots\leq\lambda_{j}(\hat\Lambda_2)\leq\cdots\leq\lambda_{j+c_2}(\hat\Lambda_2)$ are all the  eigenvalues  of $\mathcal{A}_{\hat\Lambda_2}$ in  $ [\lambda-\epsilon,\lambda+\epsilon]$.
 Since $\dim(\hat\Lambda_2\cap Gr(\gamma_\lambda(T)))= r$ and  $\lambda_i(\hat\Lambda_2)\leq\lambda_{i}(\Lambda_D)$ for  $j-c_1\leq i\leq j+c_2$, we have $\lambda_{j-(r-c_2)}(\hat\Lambda_2)<\lambda_{j-(r-c_2-1)}(\hat\Lambda_2)=\cdots=\lambda_{j}(\hat\Lambda_2)=\cdots=\lambda_{j+c_2}(\hat\Lambda_2)=\lambda$.
Therefore, $\lambda_{j}(\hat\Lambda_2)=\lambda_{j}(\Lambda_D)$.
 \end{proof}

 \section{Proof of Theorem \ref{thm:lim to Lag}}

Recall that $\ga_\lambda$ defined in (\ref{fundamental solution}) is the  fundamental solution of  \eqref{eq:Hamilton system}. In this section, we give the asymptotic behavior of $Gr(\gamma_{\lambda}(T))$ as $\lambda\to-\infty$.

\begin{proof} [Proof of Theorem \ref{thm:lim to Lag}]
  We first prove (\ref{thm:lim to Lag1}). It is equivalent to show that $\lim\limits_{\lambda\to-\infty}\mathcal{U}(Gr(\ga_\lambda(T)))=\mathcal{U}(\Lambda_D)$.
Suppose otherwise, there exist $\Lambda_0\neq \Lambda_D\in \text{Lag}(\mathcal{V},\omega_{2n})$ and a sequence  $\{v_m\}_{m=1}^\infty$ such that $\lim\limits_{m\to\infty}v_m=-\infty$ and
$\lim\limits_{m\to\infty}\mathcal{U}(Gr(\ga_{v_m}(T)))=\mathcal{U}(\Lambda_0)$.
Let $\hat\Lambda_m:=Gr(\gamma_{v_m}(T))$ for convenience.

Firstly, we claim that
\begin{equation}\label{def-K-with-claim}
K(\Sigma_{\mathcal{U}(\Lambda_0)},\Sigma_{\mathcal{U}(\Lambda_D)}):=\sup_{U_1\in\Sigma_{\mathcal{U}(\Lambda_0)}}\inf_{U_2
\in\Sigma_{\mathcal{U}(\Lambda_D)}}\|U_1-U_2\|>0.
\end{equation}
In fact, since $\mathcal{U}(\Lambda_0)\neq I_{2n}$, there exists  $\kappa\in\sigma(\mathcal{U}(\Lambda_0))$
  such that $\kappa\neq1$.  Obviously,   $\kappa I_{2n}\notin\Sigma_{\mathcal{U}(\Lambda_D)}$ and  $\kappa I_{2n}\in\Sigma_{\mathcal{U}(\Lambda_0)}$. Direct computation shows that $\inf_{U
\in\Sigma_{\mathcal{U}(\Lambda_D)}}\|\kappa I_{2n}-U\|>0$,  and thus (\ref{def-K-with-claim}) holds.

Then we claim that
\begin{equation}\label{Limit-second-claim}
K(\Sigma_{\mathcal{U}(\Lambda_0)},\Sigma_{\mathcal{U}(\hat\Lambda_m)})\to0
\end{equation}
as $m\to\infty.$

Since $\mathcal{U}(\hat\Lambda_m) \to \mathcal{U}(\Lambda_0)$, we infer that
for any $\epsilon>0$, there exists $N>0$ such that
$\|\mathcal{U}(\Lambda_0)^{-1}\mathcal{U}(\hat\Lambda_m)-I_{2n}\|<\epsilon $  for each $m>N$.
Let $U \in \Sigma_{\mathcal{U}(\Lambda_0)}$.
Then $U\mathcal{U}(\Lambda_0)^{-1}\mathcal{U}(\hat\Lambda_m)\in \Sigma_{\mathcal{U}(\hat\Lambda_m)}$.
It follows that
 $$
K(\Sigma_{\mathcal{U}(\Lambda_0)},\Sigma_{\mathcal{U}(\hat\Lambda_m)})\le \sup_{U\in \Sigma_{\mathcal{U}(\Lambda_0)}}\|U\mathcal{U}(\Lambda_0)^{-1}\mathcal{U}(\hat\Lambda_m)-U\|<\epsilon,
$$
where $m> N$. Therefore, we get \eqref{Limit-second-claim}.

By (\ref{def-K-with-claim}), there exists $\mathcal{U}(\Lambda_3)\in \Sigma_{\mathcal{U}(\Lambda_0)}$ such that $\mathcal{U}(\Lambda_3)\notin \Sigma_{\mathcal{U}(\Lambda_D)}$, and there exists a compact neighborhood  $V_{\mathcal{U}(\Lambda_3)}$ of $\mathcal{U}(\Lambda_3)$ such that $V_{\mathcal{U}(\Lambda_3)}\cap\Sigma_{\mathcal{U}(\Lambda_D)}=\emptyset$. This deduces that $\lambda_1(V_{\Lambda_3})$ is bounded from below by Proposition \ref{prop:low bound}, where $V_{\Lambda_3}=\mathcal{U}^{-1}(V_{\mathcal{U}(\Lambda_3)})$.
On the other hand,  {
we get by \eqref{Limit-second-claim} that $\lim\limits_{m\to \infty} \inf_{U\in \Sigma_{\mathcal{U}(\hat\Lambda_m)}} \|\mathcal{U}$ $(\Lambda_3)-U\|=0$.
So there exists $U_m\in \Sigma_{\mathcal{U}(\hat\Lambda_m)}$ such that $\lim\limits_{m\to \infty} \|\mathcal{U}(\Lambda_3)-U_m\|=0$.
	It follows that $\Sigma_{\mathcal{U}(\hat\Lambda_m)}\cap V_{\mathcal{U}(\Lambda_3)}\neq\emptyset$ when $m$ is sufficiently large.}
 Choose $\mathcal{U}(\tilde\Lambda_m)\in\Sigma_{\mathcal{U}(\hat\Lambda_m)}\cap V_{\mathcal{U}(\Lambda_3)}$. Then  $v_m$ is an eigenvalue of $\mathcal{A}_{\tilde\Lambda_m}$.
However, $\lim\limits_{m\to \infty}v_m=-\infty$   contradicts that $\lambda_1(V_{\Lambda_3})$ is bounded from below.

{Next, we prove (\ref{thm:lim to Lag2}).
	Using the fact that $\cA_{Gr(\gamma_{\lambda(s)}(T))}$ has an eigenvalue $\lambda(s)$ with the multiplicity to be $2n$, we obtain
that  $\cA_{Gr(\gamma_{\lambda(s)}(T))}$ has at least $2n$ eigenvalues such that they tend to $-\infty$ as $s\to(-\pi/2)^+$.
	Then we have $$\mu(\Lambda_D,Gr(\gamma_{\lambda(s)}(T)),s\in[-\pi/2,-\pi/2+\epsilon])\geq2n$$ by Theorem \ref{thm: Maslov index continu}.
	Since $Gr(\gamma_{\lambda(s)}(T))\cap {\Lambda_D}=\{0\}$ for $s\in(-\pi/2,-\pi/2+\epsilon]$ with $\epsilon>0$  small enough,
	we have $\mu(\Lambda_D,Gr(\gamma_{\lambda(s)}(T)),s\in[-\pi/2,-\pi/2+\epsilon])\leq2n$.
	It then follows that  $$\mu(\Lambda_D,Gr(\gamma_{\lambda(s)}(T)),s\in[-\pi/2,-\pi/2+\epsilon])=2n.$$ }
 \end{proof}

{ As a consequence of  Theorem \ref{thm:lim to Lag}, we get the following result.
	\begin{proposition} Let $\lambda(s)=\tan(s)$  for $s\in [-\pi/2, +\pi/2]$.
	Then there exists $\epsilon_0>0$ such that $Gr(\gamma_{\lambda(s)}(T))\cap Gr(\gamma_{\lambda(t)}(T)) =\{0\}$ for any $-\pi/2<s<t<-\pi/2+\epsilon_0$.
	\end{proposition}
	\begin{proof}
	Denote $Gr(\gamma_{\lambda(r)}(T))$ by $\Lambda_{r}$.
		Suppose that for any $\epsilon>0$, there exist $-\pi/2<s_\epsilon<t_\epsilon<-\pi/2+\epsilon$ such that $\Lambda_{s_\epsilon}\cap \Lambda_{t_\epsilon} \neq\{0\}$.
		Then $\cA_{\Lambda_{t_\epsilon}}$ has
		an eigenvalue $\lambda(t_\epsilon)$ with its multiplicity to be $2n$ and another eigenvalue $\lambda(s_\epsilon)$.
		It follows that  at least $2n+1$ eigenvalues of $\cA_{\Lambda_{t_\epsilon}}$ tend to $-\infty$ as $\epsilon\to0^+$.
		By Theorem \ref{thm:lim to Lag}, $\mu(\Lambda_D,\Lambda_{t},t\in [-\pi/2,-\pi/2+\epsilon])=2n$ with $\epsilon>0$ small enough.
		So by Theorem  \ref{thm: Maslov index continu}, the $(2n+1)$-th eigenvalue of $\cA_{\Lambda_{t}}$ is bounded from below as $t\to (-\pi/2)^+$, which is a contradiction.
		Then the proposition follows.
\end{proof}
\bigskip

\noindent {\bf Acknowledgements.}  The authors sincerely thank Prof. Yiming Long for his interest.
X. Hu is partially supported
by NSFC (Nos.11425105, 11790271). H. Zhu is partially supported  by  PITSP (No. BX20180151) and CPSF (No. 2018M630266).

\end{document}